\documentclass[letterpaper,11pt]{amsart}
\usepackage{graphicx, epigraph}
\usepackage{amsmath,amsthm,amssymb,mathtools,amsfonts,mathrsfs}
\usepackage{xspace,xcolor}
\usepackage[breaklinks,colorlinks,citecolor=teal,linkcolor=teal,urlcolor=teal,pagebackref,hyperindex]{hyperref}
\usepackage[alphabetic]{amsrefs}
\usepackage{tikz-cd}
\usepackage[all]{xy}
\usepackage{url}
\usepackage[shortlabels]{enumitem}

\newtheorem{thm}{Theorem}[section]
\newtheorem{prop}[thm]{Proposition}

\newtheorem{lem}[thm]{Lemma}
\newtheorem{cor}[thm]{Corollary}

\theoremstyle{definition}

\newtheorem{ex}[thm]{Example}
\newtheorem{rmk}[thm]{Remark}

\newtheorem{conv}{Convention}
\newtheorem*{nota*}{Notation}

\newcommand{\Z}{\mathbb Z}

\newcommand{\Pp}{\mathbb P}

\newcommand{\bM}{\overline{\mathcal M}}

\newcommand{\ualpha}{\underline\alpha}
\newcommand{\uepsilon}{\underline\varepsilon}
\newcommand{\ueta}{\underline \eta}

\newcommand{\lne}{\text{line}}
\newcommand{\conic}{\text{conic}}
\newcommand{\pt}{\text{pt}}

\newcommand{\fG}{{\mathfrak G}}

\newcommand{\vir}{\text{vir}}
\newcommand{\ev}{\operatorname{ev}}
\newcommand{\HH}{\mathfrak H}
\newcommand{\g}{\text{g}}
\newcommand{\tanglog}{K_{\text{log}}}

\title[WDVV AND ITS APPLICATION TO RELATIVE GW-THEORY]
{Witten--Dijkgraaf--Verlinde--Verlinde EQUATION AND ITS APPLICATION TO RELATIVE GROMOV--WITTEN THEORY}

\author{H.~Fan}
\email{honglu.fan@math.ethz.ch}
\author{L.~Wu}
\email{longting.wu@math.ethz.ch}

\begin{document}

\maketitle

\begin{abstract}
We derive a recursive formula for certain relative Gromov--Witten invariants with a maximal tangency condition via the Witten--Dijkgraaf--Verlinde--Verlinde equation. For certain relative pairs, we get explicit formulae of invariants using the recursive formula.

\end{abstract}


\section{Introduction}

In 1990's, one of the motivation to develop Gromov--Witten theory (GW-theory for short) is to solve enumerative problems of curves. For genus-$0$ GW-theory, the associativity of quantum cohomology, which is equivalent to the Witten--Dijkgraaf--Verlinde--Verlinde (WDVV) equation, led to Kontsevich's solution to the classical problem of counting degree $d$ rational curves passing through $3d-1$ general points in $\Pp^2$.

Similarly, we consider the enumerative problem of counting degree $d$ rational curves passing through $2d-1$ general points in $\Pp^2$ and having order $d$ with a given line at a given point.
In \cite{T} and \cite{RW}, they both showed that such an enumerative problem enjoys a similar recursion as the Kontsevich's formula. But their methods of deriving their recursion are totally different from Kontsevich's. 

In \cite{T}, Takahashi's work was based on Gathmann's definition of relative GW-invariants in genus $0$ (\cite{Ga}) and explicit calculations. In particular, Takahashi obtained recursions for $\Pp^2$ relative to a line and $\Pp^2$ relative to a conic. As for \cite{RW}, Reineke and Weist's work was based on the correspondence between relative  GW-invariants and quiver Donaldson--Thomas invariants. In fact, they obtained an explicit formula for the above enumerative problem via the correspondence and some techniques in the quiver side.

In this paper, we recover and generalize their results. We show that those recursive formulas in \cite[Theorem 3.1]{T} (with $k=1$) and \cite[Remark 1.6]{RW} can be generalized by applying the WDVV equation of relative GW-theory. The WDVV equation for the relative GW-theory is introduced in \cite{FWY}, where we enlarge the genus-0 relative GW-theory by defining relative GW-invariants with negative contact orders. Combining the recursive formula with some combinatorial methods, we also derive explicit formulae for relative pairs $(\Pp^2,\lne)$ and $(\Pp^2, \conic)$.

We remark that in \cite{B}, Reineke and Weist's result on the correspondence between relative GW-invariants and quiver Donaldson--Thomas invariants is extended to more general cases. But it is still unclear whether our recursive formula can be derived via such generalized correspondence.

From another aspect, relative invariants with maximal contact orders are closely related to local invariants by \cite{HMT}. It is natural to ask whether WDVV in local theory can be converted to WDVV in relative theory. After a conversation with Michel van Garrel and Tom Graber, we realize that the answer is affirmative, except that a slight improvement of the statement of \cite[Theorem 1.1]{HMT} needs to be made. Instead of mapping $\bM_{0,(d)}(X(\text{log}D),\beta)$ to $\bM_{0,0}(X,\beta)$, we map to $\bM_{0,1}(X,\beta)$ by remembering the maximal contact point. And the factor $d$ needs to be modified accordingly into the pullback class of $D$ along the evaluation map. 

\subsection{Main results}
The main results of this paper can be summarized as follows.

Let $X$ be a projective surface and $D$ be a smooth divisor in $X$. Let us denote $K_X+D$ by $\tanglog$. Roughly speaking, the goal is to compute the number of rational curves in $X$ with curve class $\beta$ which pass through $-\tanglog\cdot \beta-1$ general points and intersect $D$ at a \emph{given} point with contact order $D\cdot \beta$. We denote such numbers as $\bar{N}_{\beta}^{X/D}$ (see \eqref{def:fixed} for a more precise definition). We also consider similar numbers $N_{\beta}^{X/D}$ which count curves intersecting $D$ at an \emph{unspecified} point (see also \eqref{def:nfixed}).
These two invariants turn out to be related by Lemma \ref{lem:move2fix}. More precisely, we have
\[N_{\beta}^{X/D}=(D\cdot\beta)\bar{N}_{\beta}^{X/D}\]
if $-\tanglog\cdot\beta>0$. We then use the WDVV equation to show that $N_{\beta}^{X/D}$ satisfies the following recursion:
\begin{thm}[= Theorem \ref{thm:recurMax}]
Using notations as above, we further assume that $D$ is ample and $-\tanglog\cdot\beta\geq 3$. Then we have 
\begin{equation*}
    \begin{split}
        &(H\cdot D)\frac{N_\beta^{X/D}}{d}= \\
        & \sum\limits_{\beta_1+\beta_2=\beta} \left(d_1^2(H\cdot\beta_2){-\tanglog\cdot\beta-3\choose -\tanglog\cdot\beta_1 -1}+d_1d_2(H\cdot\beta_2){-\tanglog\cdot\beta-3\choose -\tanglog\cdot\beta_1-2} \right.\\
        &-\left. d_1^2(H\cdot\beta_1){-\tanglog\cdot\beta-3 \choose -\tanglog\cdot\beta_1} - d_1d_2(H\cdot \beta_1){-\tanglog\cdot\beta-3 \choose -\tanglog\cdot\beta_1-1}\right)\frac{N_{\beta_1}^{X/D}}{d_1}\frac{N_{\beta_2}^{X/D}}{d_2},
    \end{split}
\end{equation*}
where $H$ is any divisor, $d=D\cdot\beta$, $d_1=D\cdot\beta_1$, $d_2=D\cdot\beta_2$, the summation takes over $\beta_1\neq 0$, $\beta_2\neq 0$ and $*\choose *$ are binomial coefficients.
\end{thm}
We remark that the binomial coefficients $a\choose b$ in the recursion allow the cases $a<b$ or $b<0$. In these cases, we simply set it to be zero. The condition that $D$ is ample is crucial. Without this condition, extra relative GW-invariants may appear and are more complicated to deal with. In this case, we would not get a recursive formula only involving $N_\beta^{X/D}$.

Using the above recursion, we then get explicit formulae for relative GW-invariants of pairs $(\Pp^2,\lne)$ and $(\Pp^2, \conic)$ (see Theorem \ref{thm:closefm} and Corollary \ref{cor:explfm} for more details). For the relative pair $(\Pp^2,\lne)$, it recovers some results in \cite{RW} via a totally different method.

\subsection{Outline of the paper}
The paper is organized as follows. In Section \ref{sec:review}, we provide a brief introduction of relative GW-invariants with negative contact orders and the WDVV equation. In Section \ref{sec:appl}, we use the WDVV equation to derive our main recursion and explicit formulae for relative pairs $(\Pp^2,\lne)$ and $(\Pp^2,\conic)$.

\subsection{Acknowledgment}
We thank Pierrick Bousseau, Rahul Pandharipande, Michel van Garrel and Tom Graber for helpful discussions. We also thank Fenglong You for collaborating on a related project. A special thank to the anonymous referees for pointing out the work of Takahashi which we were not aware of. H. F. is supported by grant ERC-2012-AdG-320368-MCSK and SwissMAP. L. W. is supported by grant ERC-2017-AdG-786580-MACI.

This project has received funding from the European Research Council (ERC) under the European Union’s
Horizon 2020 research and innovation program (grant agreement No. 786580).

\section{Relative GW-Invariants}\label{sec:review}
\subsection{The general theory}
In \cite{FWY}, genus-0 relative GW-invariants with negative contact orders are defined. We briefly outline the definitions in this subsection. We do not provide a detailed account of the whole theory because only special cases in Section \ref{subsec:special} (relative GW-invariants with $0$ and $1$ negative marking) are needed for computations in Section \ref{sec:appl}.

Let $X$ be a smooth projective variety and $D$ a smooth divisor. We define a \emph{topological type} $\Gamma$ to be a tuple $(g,n,\beta,\rho,\vec{\mu})$ where $g,n$ are non-negative integers, $\beta\in H_2(X,\mathbb{Z})$ is a curve class and $\vec{\mu}=(\mu_1,\dotsc,\mu_\rho)\in \Z^\rho$ is a partition of the number $\int_\beta D$. We will focus on $g=0$ case in this paper.

For each $\Gamma$, we associate it with a set $\mathcal B_\Gamma$ of connected bipartite admissible graphs of topological type $\Gamma$(see \cite{FWY}*{Definition 4.8}). A genus-0 bipartite admissible graph is a tuple
\[\fG=(\{\Gamma_i^0\},\Gamma^\infty,I,E,\g=0,b)\in \mathcal B_\Gamma.\]
These notations can be briefly summarized as follows. Each $\Gamma^0_i$ is an admissible graph encoding topological data of relative stable maps to non-rigid targets. $\Gamma^{\infty}$ is a (possibly) disconnected admissible graph encoding topological data of relative stable maps for the pair $(X,D)$. $E$ is the set of edges (which later records the gluing of relative stable maps), and $I, \g, b$ assign marked points, genera and curve classes to components of curves corresponding to vertices in $\bigcup_{\Gamma^0_i} V(\Gamma^0_i)\cup V(\Gamma^\infty)$, respectively. Since we set $g=0$ throughout this paper, the genus assignment $\g$ is a zero map.
We then use the above data to glue moduli spaces as the following:
\[\bM_{\fG} = \prod\limits_{\Gamma^0_i}\bM_{\Gamma^0_i}^\sim(D) \times_{D^{|E|}} \bM_{\Gamma^\infty}^{\bullet}(X,D)\]
where $\bM_{\Gamma^0_i}^\sim(D)$ is the moduli space of relative stable maps to rubber targets (see \cite{FWY}*{Section 2.2}) and $\bM_{\Gamma^\infty}^{\bullet}(X,D)$ is the moduli space of relative stable maps with possibly disconnected domain curves (see \cite{FWY}*{Section 2.1}).

The fiber product identifies evaluation maps according to edges. We have a natural diagram
\begin{equation*}
\xymatrix{
\bM_{\fG} \ar[r]^{} \ar[d]^{\iota} & D^{|E|} \ar[d]^{\Delta} \\
\prod\limits_{\Gamma^0_i}\bM_{\Gamma^0_i}^\sim(D) \times \bM_{\Gamma^\infty}^{\bullet}(X,D) \ar[r]^{} & D^{|E|}\times D^{|E|}
}
\end{equation*}
and a natural virtual class 
\[
[\bM_{\fG}]^{\vir}={\Delta}^![\prod\limits_{\Gamma^0_i}\bM_{\Gamma^0_i}^\sim(D)\times \bM_{\Gamma^\infty}^{\bullet}(X,D)]^{\vir}.
\]

In \cite{FWY}, the key step to define relative GW-invariants with negative contact orders is to construct a cycle $\mathfrak c_\Gamma(X/D)$ in $\bM_{0,n+\rho}(X,\beta)\times_{X^{\rho}}D^{\rho}$. Actually, we first define $\mathfrak c_\Gamma(X/D)$ as the limit of the virtual cycles in orbifold GW-theory. We then define $\mathfrak c_\Gamma(X/D)$ using cycles in relative GW-theory and show that two definitions coincide. In this section, we only give a brief description of our 2nd definition. But we need to keep in mind that the WDVV equation which we will use in this paper (see Proposition \ref{prop:WDVV}) actually follows from our 1st definition.

Firstly, for each graph $\fG\in \mathcal B_\Gamma$, we construct a cycle $C_{\fG}$ in the Chow group of $\prod\limits_{\Gamma^0_i}\bM_{\Gamma^0_i}^\sim(D)\times \bM_{\Gamma^\infty}^{\bullet}(X,D)$. We define
\[\mathfrak c_\Gamma(X/D) = \sum\limits_{\fG \in \mathcal B_\Gamma} \dfrac{1}{|Aut(\fG)|}(\mathfrak t_{\fG})_* ({\iota}^* C_{\fG} \cap [\bM_{\fG}]^{\vir}) \in A_*(\bM_{0,n+\rho}(X,\beta)\times_{X^{\rho}}D^{\rho})\]
where 
\[
\mathfrak t_{\fG}:\bM_{\fG}\rightarrow \bM_{0,n+\rho}(X,\beta)\times_{X^{\rho}}D^{\rho}
\]
is the stabilization map gluing curves according to edges $E$ and contracting bubbles of targets. $Aut(\fG)$ is the automorphism group of the graph $\fG$. The cycle $\mathfrak c_\Gamma(X/D)$ is of pure dimension
\begin{equation}\label{eqn:virdim}
d=\mathrm{dim}(X)-3+\int_{\beta} c_1(T_X(-\mathrm{log} D)) + n + \rho_+
\end{equation}
where $\rho_+$ is the number of positive integers in $\vec{\mu}=(\mu_1,\dotsc,\mu_\rho)$. Using $\mathfrak c_\Gamma(X/D)$, relative GW-invariants with negative contact orders can be defined as follows.

There are evaluation maps from $\bM_\fG$ corresponding to interior markings and relative markings
\begin{align*}
\ev_X=(\ev_{X,1},\ldots,\ev_{X,n}):\bM_\fG&\rightarrow X^n, \\
\ev_D=(\ev_{D,1},\ldots,\ev_{D,\rho}):\bM_\fG&\rightarrow D^\rho.
\end{align*}
There are also evaluation maps
\begin{align*}
\overline{\ev}_X=(\overline\ev_{X,1},\ldots,\overline\ev_{X,n}):\bM_{0,n+\rho}(X,\beta)\times_{X^{\rho}}D^{\rho}&\rightarrow X^n, \\
\overline{\ev}_D=(\overline\ev_{D,1},\ldots,\overline\ev_{D,\rho}):\bM_{0,n+\rho}(X,\beta)\times_{X^{\rho}}D^{\rho}&\rightarrow D^\rho,
\end{align*}
such that $\overline{\ev}_X\circ t_\fG=\ev_X, \overline{\ev}_D\circ t_\fG=\ev_D$. 
Now let
\begin{align*}
    \begin{split}
        \ualpha = (\alpha_1,\ldots,\alpha_n) 
        &\in H^*(X)^{\otimes n},\\ 
        \uepsilon = (\epsilon_1,\ldots,\epsilon_\rho) &\in H^*(D)^{\otimes\rho}
    \end{split}
\end{align*}
be the insertions. 

We define \emph{relative GW-invariant} of topological type $\Gamma$ with insertions $\uepsilon, \ualpha$ via  the following integral over the cycle $\mathfrak c_\Gamma(X/D)$:
\begin{equation}\label{def:relGWI}
\langle \uepsilon \mid \ualpha \rangle_{\Gamma}^{(X,D)} := \displaystyle\int_{\mathfrak c_\Gamma(X/D)} \prod\limits_{j=1}^{\rho} \overline{\ev}_{D,j}^*\epsilon_j\prod\limits_{i=1}^n \overline{\ev}_{X,i}^*\alpha_i.
\end{equation}

\subsection{Special cases}\label{subsec:special}
In this subsection, we provide a detailed account of relative GW-invariants with $0$ or $1$ negative marking, which will be important to our application in the next section. For simplicity, we assume that all the cohomology classes involved in this subsection are even degree.

\begin{ex}[Relative GW-invariants without negative markings]
In this case, our relative GW-cycle is simply the pushforward of the virtual cycle of the moduli space in the sense of \cite{Jun1}(\cite[Example 5.4]{FWY}). Under our notations,
\[\mathfrak c_\Gamma(X/D)=(\mathfrak t_{\fG})_*([\bM_\Gamma(X,D)]^{\vir})\]
where $\fG$ is the graph without vertices $\Gamma^0_i$ ($\{\Gamma_i^0\}=\emptyset$ and $\Gamma^\infty=\Gamma$). $t_{\fG}$ is simply the stabilization map
\[
\mathfrak t_{\fG}:\bM_\Gamma(X,D)\rightarrow \bM_{0,n+\rho}(X,\beta)\times_{X^{\rho}}D^{\rho}.
\]
In particular, relative GW-invariants \eqref{def:relGWI} in this case is the usual relative GW-invariants.
\end{ex}

\begin{ex}[Relative GW-invariants with $1$ negative marking]\label{ex:1neg}
We have shown in \cite[Example 5.5]{FWY} that, in this case only those graphs $\fG=(\{\Gamma_i^0\},\Gamma^\infty,I,E,\g=0,b)$
such that $\{\Gamma_i^0\}$ consists of only one element (denoted by $\Gamma^0$) can appear. 

Those graphs form a subset of $\mathcal B_\Gamma$. We denote it as $\mathcal B_{\Gamma}'$. For each $\fG\in B_{\Gamma}'$, it is shown that $C_{\fG}=\prod_{e\in E}d_e$, where $d_e$ is the multiplicity associated to the edge $e$.
So 
\[\mathfrak c_\Gamma(X/D) = \sum\limits_{\fG \in \mathcal B_\Gamma} \dfrac{\prod_{e\in E}d_e}{|Aut(\fG)|}(\mathfrak t_{\fG})_* ([\bM_{\fG}]^{\vir}).\]
Here, $[\bM_{\fG}]^{\vir} = \Delta^![\bM_{\Gamma^0}^\sim(D)\times \bM_{\Gamma^\infty}^{\bullet}(X,D)]^{\vir}$. 
This formula is not yet ready for explicit calculations in this paper. We describe the corresponding relative GW-invariants more explicitly as follows.

As before, we are given insertion vectors $\underline\alpha, \underline\varepsilon$. Without loss of generality, we assume that $\epsilon_1$ is the insertion corresponding to the unique negative marking. Other markings are either assigned to $\Gamma^0$ or $\Gamma^\infty$, and we split up insertions accordingly. We divide $\ualpha$ into $\ualpha_0$, $\ualpha_{\infty}$ and $\uepsilon$ into $\uepsilon_0$, $\uepsilon_{\infty}$ such that $\ualpha_0,\uepsilon_0$ correspond to markings assigned to $\Gamma^0$, and $\ualpha_\infty, \uepsilon_\infty$ correspond to markings assigned to $\Gamma^{\infty}$. Now the 
relative GW-invariant can be written as follows.
\begin{align*}
&\langle \uepsilon \mid \ualpha \rangle_\Gamma^{(X,D)} \\
=&\sum\limits_{\fG \in \mathcal B_\Gamma} \frac{\prod_{e\in E}d_e}{|Aut(E)|}\sum_{\eta}\langle\uepsilon_0\mid \ualpha_0\mid \ueta,\epsilon_1\rangle^{\sim}_{\Gamma^0}\langle \check{\ueta},\uepsilon_{\infty}\mid \ualpha_{\infty}\rangle^{\bullet, (X,D)}_{\Gamma^\infty},
\end{align*}
where $Aut(E)$ is the permutation group of the set $\{d_1,d_2,\ldots,d_{|E|}\}$, $\eta$ ranges over a basis of $H^*(D)^{\otimes |E|}$ ($\check{\ueta}$ is the dual basis of $\ueta$), and \[\langle\uepsilon_0\mid \ualpha_0\mid \ueta,\epsilon_1\rangle^{\sim}_{\Gamma^0},\quad \langle \check{\ueta},\uepsilon_{\infty}\mid \ualpha_{\infty}\rangle^{\bullet, (X,D)}_{\Gamma^\infty}\]
are rubber and relative invariants defined by integration over the virtual cycle of $\bM_{\Gamma^0}^\sim(D)$, $\bM_{\Gamma^\infty}^{\bullet}(X,D)$  separately. The subscript $\bullet$ indicates that relative GW-invariant of the pair $(X,D)$ is possibly disconnected, which can be computed by multiplying the individual connected invariants.
\end{ex}

\subsection{WDVV equation}
By introducing negative contact orders to relative GW-theory, WDVV equation for relative GW-theory can be expressed in a nicer way (see also \cite[Proposition 7.5]{FWY}). We summarize this formula in this subsection.
\begin{conv}
Let $\{T_k\}$ be a basis of $H^*(X)$ and $\{\bar T_k\}$ be a basis for $H^*(D)$. We use $\{T^k\}$ (resp. $\{\bar T^k\}$) to denote the dual basis of $\{T_k\}$ (resp. $\{\bar T_k\}$). Let $\HH_0=H^*(X)$ and $\HH_i=H^*(D)$ if $i\in \Z - \{0\}$. The space of insertions for relative GW-invariants is given by 
\[
\HH=\bigoplus\limits_{i\in\Z}\HH_i.
\]
The index here corresponds to contact orders in relative GW-theory. For $\alpha\in \HH_i$, we denote its natural image in $\HH$ by $[\alpha]_i$. Using $\{T_k\}$ and $\{\bar T_k\}$, we can construct a natural basis of $\HH$ by
\begin{align*}
\widetilde T_{0,k}&=[T_k]_0, \\
\widetilde T_{i,k}&=[\bar T_k]_i \text{ when } i\neq 0.
\end{align*}
Let 
\[
\begin{split}
([\alpha]_i,[\beta]_j) = 
\begin{cases}
0, &\text{if } i+j\neq 0,\\
\int_X \alpha\cup\beta, &\text{if } i=j=0, \\
\int_D \alpha\cup\beta, &\text{if } i+j=0, i,j\neq 0
\end{cases}
\end{split}
\]
be a natural pairing on $\HH$. Under the such paring, the dual basis of $\{\widetilde T_{i,k}\}$ is given by $\{\widetilde T_{-i}^k\}$ where
\begin{align*}
\widetilde T_{0}^k&=[T^k]_0, \\
\widetilde T_{i}^k&=[\bar T^k]_i \text{ when } i\neq 0.
\end{align*}
\end{conv}

Using the above convention, we rewrite 
\[\langle \uepsilon \mid \ualpha \rangle_{\Gamma}^{(X,D)}\]
as 
\[I_{\beta}([\epsilon_1]_{\mu_1},\cdots,[\epsilon_{\rho}]_{\mu_{\rho}},[\alpha_1]_0,\cdots,[\alpha_n]_0).\]
Here, we omit the relative pair $(X,D)$ for simplicity.

Now we are ready to state the WDVV equation for relative GW-theory.
\begin{prop}[WDVV]\label{prop:WDVV}
\begin{align*}
    &\sum I_{\beta_1}([\alpha_1]_{i_1}, [\alpha_2]_{i_2}, \prod\limits_{j\in S_1} [\alpha_j]_{i_j}, \widetilde T_{i,k}) I_{\beta_2}(\widetilde T_{-i}^k, [\alpha_3]_{i_3}, [\alpha_4]_{i_4}, \prod\limits_{j\in S_2} [\alpha_j]_{i_j}) \\
    =&\sum I_{\beta_1}([\alpha_1]_{i_1}, [\alpha_3]_{i_3}, \prod\limits_{j\in S_1} [\alpha_j]_{i_j}, \widetilde T_{i,k}) I_{\beta_2}(\widetilde T_{-i}^k, [\alpha_2]_{i_2}, [\alpha_4]_{i_4}, \prod\limits_{j\in S_2} [\alpha_j]_{i_j}),
\end{align*}
where each sum is over all $\beta_1+\beta_2=\beta$, all indices $i,k$ of basis, and $S_1, S_2$ disjoint sets with $S_1\cup S_2=\{5,\ldots,m\}$. Also, the $\prod$ symbol makes each factor as a separate insertion, instead of multiplying them up.
\end{prop}

\begin{rmk}
In ordinary GW-theory, the WDVV equation can be obtained by pulling back the cycle relation on $\bM_{0,4}$. As for the relative GW-theory, if we pull back the same cycle relation to the moduli space of relative stable map, we expect to get exactly the same relation as above with $i_n\geq 0$ for all $1\leq n\leq m$. This is a part of an ongoing project.
\end{rmk}

\section{Application of the WDVV equation}\label{sec:appl}
In this section, we focus on a relative pair $(X,D)$ such that $X$ is a smooth projective surface and $D$ is an ample divisor. We mainly consider the following two kinds of relative invariants:
\begin{eqnarray}
N_{\beta}^{X/D}: & = & I_{\beta}([1_D]_{D\cdot \beta},\underbrace{[\omega_X]_0,\cdots,[\omega_X]_0}_{n})\label{def:nfixed}\\
\bar{N}_{\beta}^{X/D}: & = & I_{\beta}([\omega_D]_{D\cdot \beta},\underbrace{[\omega_X]_0,\cdots,[\omega_X]_0}_{\bar{n}})\label{def:fixed}
\end{eqnarray}
where $\omega_{X}\in H^4(X), \omega_D\in H^2(D)$ are the Poincar\'e duals of point classes of $X$ and $D$, respectively. $1_D$ is the identity element in $H^*(D)$. By the virtual dimension \eqref{eqn:virdim}, we must have 
\[n=-\tanglog\cdot\beta,\quad \bar{n}=-\tanglog\cdot\beta-1,\]
where 
\[
\tanglog=K_X+D
\]
and $\tanglog\cdot\beta$ is the intersection number of divisor $\tanglog$ and curve class $\beta$. We will apply the WDVV equation to get a recursive formula for $N_{\beta}^{X/D}$ (or $\bar{N}_{\beta}^{X/D}$).

First of all, we want to show that $N_{\beta}^{X/D}$ and $\bar{N}_{\beta}^{X/D}$ are in fact closely related by the following lemma. As a side note, we remark that when $X=\mathbb{P}^2$ and $D=\lne$, Lemma \ref{lem:move2fix} already appeared in \cite[Lemma 2.7]{T} and \cite[Corollary 4.8]{FM}. \cite[Lemma 2.7]{T} also includes the case when $X=\mathbb{P}^2$ and $D=\conic$.

\begin{lem}\label{lem:move2fix}
If $-\tanglog\cdot\beta>0$, we have
\[N_{\beta}^{X/D}=(D\cdot\beta)\bar{N}_{\beta}^{X/D}.\]
\end{lem}

\begin{proof}
We will prove it via degeneration formula \cites{EGH,IP2,Jun2,LR}.
Using deformation to the normal cone, we obtain the degeneration 
\[X\rightsquigarrow X\cup_D \mathbb{P}(N_{D/X}\oplus \mathcal{O})\]
where $N_{D/X}$ is the normal bundle of $D$ in $X$. There are two obvious $\mathbb{C}^*$-invariant divisors $D_0$ and $D_{\infty}$ on the $\mathbb{P}^1$-bundle $P=\mathbb{P}(N_{D/X}\oplus \mathcal{O})$, whose normal bundle is given by $N_{D/X}$ and $N_{D/X}^{\vee}$ respectively. $X$ is glued to $P$ via $D_{\infty}$.

Now by the degeneration formula, $N_{\beta}^{X/D}$ can be determined from relative invariants of the pairs $(X,D)$ and $(P, D_0\cup D_{\infty})$.
We then distribute original $-\tanglog\cdot \beta$ point conditions in $N_{\beta}^{X/D}$ to $-\tanglog\cdot \beta-1$ point conditions on $X$ and one point condition on $P$.

Now each term in the degeneration formula can be written in the following form:
\begin{equation*}
C_{\Gamma_1,\Gamma_2}\langle 1_D\mid[\pt]\mid  \ueta \rangle^{\bullet,
 (P, D_0\cup D_{\infty})}_{\Gamma_1}\langle\check{\ueta}\mid\underbrace{[\omega_X]_0,\cdots,[\omega_X]_0}_{-\tanglog\cdot \beta-1}\rangle^{\bullet,(X,D)}_{\Gamma_2}
\end{equation*}
where $C_{\Gamma_1,\Gamma_2}$ is some constant determined by $\Gamma_1,\Gamma_2$ and $[\pt]$ is Poincar\'e dual to the point class of $P$.

Since $D$ is an ample divisor in $X$, for each effective curve class $\beta_D$ in $D$, we must have
$\int_{\beta_D} c_1(N_{D/X})\geq 0$. Let $v$ be a vertex in $\Gamma_1$ whose associated curve class is $\beta_v$. Then we must have 
\[D_0\cdot\beta_v-D_{\infty}\cdot \beta_v=\int_{\pi_*(\beta_v)}c_1(N_{D/X})\geq 0,\quad D_0\cdot\beta_v\geq 0,\quad D_{\infty}\cdot\beta_v\geq 0\]
where $\pi:P\rightarrow D$ be the projection map.
In the degeneration formula, we require that the gluing of $\Gamma_1$ and $\Gamma_2$ gives a connected graph. So $\beta_v$ must have an positive intersection number with $D_{\infty}$. This further implies that $D_0\cdot\beta_v>0$. But the intersection profile given on $D_0$ consists of one point with maximal tangency. This forces $\Gamma_1$ to have only one vertex. 

Let $p:\bM_{\Gamma_1}(P,D_0\cup D_{\infty})\rightarrow \bM_{0,m}(D,\pi_*(\beta_1))$ be the natural morphism induced by the projection $\pi$ and stabilization process. Here $m$ is the total number of interior and relative marking points
associated to $\Gamma_1$ and $\beta_1$ is the curve class. Combining the rigidification lemma \cite[Lemma 2]{MP} with \cite{JPPZ}*{Theorem 2},
we may deduce that
\[p_*\left(ev^*(D_{\infty})\cap [\bM_{\Gamma_1}(P,D_0\cup D_{\infty})]^{\vir}\right)=[\bM_{0,m}(D,\pi_*(\beta_1))]^{\vir}\]
where $ev$ is the evaluation map associated to the unique interior marking. We remark that the above equation only works for genus $0$.

Now by the projection formula and the string equation, we know that 
\[\langle 1_D\mid[\pt]\mid  \ueta \rangle^{\bullet,
 (P, D_0\cup D_{\infty})}_{\Gamma_1}\]
does not vanish only when $\pi_*(\beta_1)=0$ and there are three markings in total (one maps to $D_0$, one maps to $D_{\infty}$ and one interior marking). So the topological type of $\Gamma_2$ is also fixed in this case. We compute that the contribution of such triple $(\Gamma_1,\Gamma_2,I)$ is $(D\cdot\beta)\bar{N}_{\beta}^{X/D}$. 
\end{proof}

\subsection{A recursive formula}
Now we are ready to show that the WDVV equation can be used to solve $N_{\beta}^{X/D}$ up to some initial conditions. Let $H$ be any divisor, we have the following identity.
\begin{thm}\label{thm:recurMax}
For any smooth projective surface $X$, smooth ample divisor $D$ and curve class $\beta$ such that $-\tanglog\cdot \beta\geq 3$, we have the following recursive relation:
\begin{equation}\label{eqn:max}
    \begin{split}
        &(H\cdot D)\frac{N_\beta^{X/D}}{d}= \\
        & \sum\limits_{\beta_1+\beta_2=\beta} \left(d_1^2(H\cdot\beta_2){-\tanglog\cdot\beta-3\choose -\tanglog\cdot\beta_1 -1}+d_1d_2(H\cdot\beta_2){-\tanglog\cdot\beta-3\choose -\tanglog\cdot\beta_1-2} \right.\\
        &-\left. d_1^2(H\cdot\beta_1){-\tanglog\cdot\beta-3 \choose -\tanglog\cdot\beta_1} - d_1d_2(H\cdot \beta_1){-\tanglog\cdot\beta-3 \choose -\tanglog\cdot\beta_1-1}\right)\frac{N_{\beta_1}^{X/D}}{d_1}\frac{N_{\beta_2}^{X/D}}{d_2},
    \end{split}
\end{equation}
where $d=D\cdot\beta$ $d_1=D\cdot\beta_1$, $d_2=D\cdot\beta_2$, the summation takes over $\beta_1\neq 0$, $\beta_2\neq 0$ and $*\choose *$ are binomial coefficients.
\end{thm}
\begin{proof}
We will prove \eqref{eqn:max} using the WDVV equation in Proposition \ref{prop:WDVV}. We use the following insertions:
\[[\alpha_1]_{i_1}=[1_D]_{D\cdot \beta},\, [\alpha_2]_{i_2}=[H]_0,\,[\alpha_s]_{i_s}=[\omega_X]_0,\quad 3\leq s \leq (-\tanglog,\beta)+1.\]
Note that there is only one marking with positive contact order. The rest of the markings are interior (i.e., with contact orders equal to zero).

Let us consider the LHS of the WDVV equation at first:
\[\sum I_{\beta_1}([1_D]_{D\cdot \beta}, [H]_0, \prod\limits_{j\in S_1} [\omega_X]_0, \widetilde T_{i,k}) I_{\beta_2}(\widetilde T_{-i}^k, [\omega_X]_0, [\omega_X]_0, \prod\limits_{j\in S_2} [\omega_X]_0).\]

Suppose that $\widetilde T_{i,k}=[T_k]_0$. This insertion corresponds to an interior marking. Since the sum of all contact orders equals $D\cdot\beta_1$, we conclude that $D\cdot\beta_2=0$. The ampleness of $D$ further implies that $\beta_2=0$. Now since there are two point-insertions in $I_{\beta_2=0}$, we must have $I_{\beta_2}=0$.

Next, suppose that $\widetilde T_{i,k}=[\bar T_k]_i$, $i\neq 0$. Now the sum of all contact orders becomes $D\cdot\beta+i$ which equals $D\cdot \beta_1$. So $i=-D\cdot\beta_2$. Since $D$ is ample, we must have $i< 0$. The dimension constraint further requires $\bar T_k$ to be of even degree. Since $\dim_{\mathbb{C}}(D)=1$, we deduce that either $\bar T_k=1_D$ or $\bar T_k=\omega_D$. We discuss these two cases separately.

\begin{enumerate}[(I)]
\item If $\bar T_k=1_D$, then $\widetilde T_{-i}^k=[\omega_D]_{-i}$. In this case, it is easy to see that $I_{\beta_2}=\bar N_{\beta_2}^{X/D}$. As for $I_{\beta_1}$, it only contains one negative marking. So according to the discussion in Example \ref{ex:1neg}, we know that $I_{\beta_1}$ is a summation of the following terms:
\begin{equation}\label{eqn:oneneg}
\frac{\prod_{e\in E}d_e}{|Aut(E)|}\sum_{\eta}\langle1_D\mid \ualpha_0\mid \ueta,1_D\rangle^{\sim}_{\Gamma^0}\langle \check{\ueta}\mid \ualpha_{\infty}\rangle^{\bullet, (X,D)}_{\Gamma^\infty},
\end{equation}
where $\ualpha_0$ and $\ualpha_{\infty}$ is a division of insertions $[H]_0, \prod\limits_{j\in S_1} [\omega_X]_0$. Similar to the proof of Lemma \ref{lem:move2fix}, we can show that the rubber invariant
\[\langle1_D\mid \ualpha_0\mid \ueta,1_D\rangle^{\sim}_{\Gamma^0}\]
does not vanish only when the push-forward of curve class to the base is zero and there are totally three markings. 

\begin{enumerate}
\item If $\ueta$ is empty which only happens when $\beta_1=0$, $\ualpha_0=H$ and $S_1$ is empty. In this case, we have
\[I_{\beta_1=0}=\langle1_D\mid H\mid 1_D\rangle^{\sim}_{\Gamma^0}=H\cdot D,\qquad I_{\beta_2=\beta}=\bar N_{\beta}^{X/D}.\]
So the total contribution of case (a) to the LHS of the WDVV equation is
\[(H\cdot D)\frac{N_{\beta}^{X/D}}{(D\cdot\beta)}.\]
Here we have used Lemma \ref{lem:move2fix}.

\item If $\ueta$ is not empty, $\ueta$ must contain only one insertion and $\ualpha_0$ must be empty. By dimensional constraint, $\ueta=\omega_D$. Recall that the insertion $\ueta$ corresponds to a relative marking over $D_\infty$ with contact order $D\cdot \beta_1$. So there is only one edge $e$ with degree $d_e=D\cdot\beta_1$. From the above discussion, only the following type of invariants contributes to \eqref{eqn:oneneg}:
\begin{align*}&(D\cdot\beta_1)I_{\beta_1}([1_D]_{D\cdot \beta_1},[H]_0,\underbrace{[\omega_X]_0,\cdots,[\omega_X]_0}_{-\tanglog\cdot\beta_1})\\
=&(D\cdot\beta_1)(H\cdot \beta_1)N_{\beta_1}^{X/D}.
\end{align*}
Note that in the above expression, we have the freedom to choose $-\tanglog\cdot \beta_1$ markings from the original $-\tanglog\cdot \beta-3$ markings in order to insert $[\omega_X]_0$ (thus ${-\tanglog\cdot\beta-3\choose -\tanglog\cdot\beta_1}$ choices). In case (b), we also have
$I_{\beta_2}=\bar N_{\beta_2}^{X/D}$.

So the total contribution of case (b) to the LHS of the WDVV equation is
\[\sum_{\beta_1+\beta_2=\beta}(D\cdot\beta_1)^2(H\cdot \beta_1){-\tanglog\cdot\beta-3\choose -\tanglog\cdot\beta_1}\frac{N_{\beta_1}^{X/D}}{(D\cdot \beta_1)}\frac{N_{\beta_2}^{X/D}}{(D\cdot\beta_2)}.\]
\end{enumerate} 

Combining cases (a) and (b), we conclude that the total contribution of case (I) to the LHS of the WDVV equation is
\begin{align*}
(H\cdot D)\frac{N_{\beta}^{X/D}}{(D\cdot\beta)}+\sum_{\beta_1+\beta_2=\beta}&(D\cdot\beta_1)^2(H\cdot \beta_1){-\tanglog\cdot\beta-3\choose -\tanglog\cdot\beta_1}\\
&\cdot \frac{N_{\beta_1}^{X/D}}{(D\cdot \beta_1)}\frac{N_{\beta_2}^{X/D}}{(D\cdot\beta_2)}.
\end{align*}

\item If $\bar T_k=\omega_D$, then $\widetilde T_{-i}^k=[1_D]_{-i}$. Since the computation is similar to that of case (I), we omit the details here. The total contribution is \[\sum_{\beta_1+\beta_2=\beta}(H\cdot \beta_1){-\tanglog\cdot\beta-3\choose -\tanglog\cdot\beta_1-1}N_{\beta_1}^{X/D}N_{\beta_2}^{X/D}.\]
Combing cases (I) and (II), we conclude that the LHS of the WDVV equation is

\begin{align*}
&(H\cdot D)\frac{N_{\beta}^{X/D}}{(D\cdot\beta)}\\
&+\sum_{\beta_1+\beta_2=\beta} \left((D\cdot\beta_1)^2(H\cdot \beta_1){-\tanglog\cdot\beta-3\choose -\tanglog\cdot\beta_1}\frac{N_{\beta_1}^{X/D}}{(D\cdot \beta_1)}\frac{N_{\beta_2}^{X/D}}{(D\cdot\beta_2)}\right.\\
&\left.+(H\cdot \beta_1){-\tanglog\cdot\beta-3\choose -\tanglog\cdot\beta_1-1}N_{\beta_1}^{X/D}N_{\beta_2}^{X/D}\right).
\end{align*}

\end{enumerate}

Similarly, the RHS of the WDVV equation is
\[
\begin{split}
    \sum_{\beta_1+\beta_2=\beta}&\left((D\cdot\beta_1)^2(H\cdot \beta_2){-\tanglog\cdot\beta-3\choose -\tanglog\cdot\beta_1-1}\frac{N_{\beta_1}^{X/D}}{(D\cdot \beta_1)}\frac{N_{\beta_2}^{X/D}}{(D\cdot\beta_2)}\right.\\
& \left.+(H\cdot \beta_2){-\tanglog\cdot\beta-3\choose -\tanglog\cdot\beta_1-2}N_{\beta_1}^{X/D}N_{\beta_2}^{X/D}\right).
\end{split}
\]
Now Theorem \ref{thm:recurMax} directly follows.
\end{proof}

When $D$ is not ample, the WDVV equation still holds. But the computation of $N_{\beta}^{X/D}$ involves some relative invariants that are not directly accessible. Therefore, it will not produce a recursive formula only involving $N_{\beta}^{X/D}$ or $\bar N_{\beta}^{X/D}$.

In Theorem \ref{thm:recurMax}, if we choose $H$ to be an ample divisor, we can further conclude the following.
\begin{cor}
The set of invariants
\[
\{N_{\beta}^{X/D}\mid -\tanglog\cdot\beta\geq 3\}
\]
can be recursively determined by the set of invariants 
\[
\{N_{\beta}^{X/D}\mid -\tanglog\cdot\beta< 3\}.
\]
\end{cor}

\subsection{Some examples}
In this subsection, we show that for pairs $(\mathbb{P}^2,\lne)$ and $(\mathbb{P}^2,\conic)$, Theorem \ref{thm:recurMax} can be used to calculate explicit formulae of $N_{\beta}^{X/D}$. The 1st several $N_{\beta}^{X/D}$ will be computed by Gathmann's program GROWI \cite{GROWI}.

Let us set $H=D$ in \eqref{eqn:max}. And using Lemma \ref{lem:move2fix}, we get
\begin{equation}\label{eqn:simprec}
\begin{split}
        (D\cdot D)\bar{N}_{\beta}^{X/D}=\sum\limits_{\beta_1+\beta_2=\beta} &\left((D\cdot\beta_1)(D\cdot\beta_2)^2{-\tanglog\cdot\beta-3\choose -\tanglog\cdot\beta_1-2}\right.\\
        & \left.-(D\cdot\beta_1)^3{-\tanglog\cdot\beta-3 \choose -\tanglog\cdot\beta_1} \right)\bar{N}_{\beta_1}^{X/D}\bar{N}_{\beta_2}^{X/D}.
    \end{split}
\end{equation}

For simplicity, relative invariants $\bar{N}_{\beta}^{X/D}$ for pairs $(\mathbb{P}^2,\lne)$ and $(\mathbb{P}^2,\conic)$ are distinguished by adding a superscript $L$ or $C$. The curve class $\beta$ is replaced by its corresponding degree $d\in \mathbb{Z}_{>0}$. For example, $\bar{N}_d^L$ stands for relative invariant $\bar{N}_{\beta}^{\mathbb{P}^2/\lne}$ with degree $d$.

Let
\begin{eqnarray*}
F^L &=& \sum_{d=1}^{\infty}\frac{\bar{N}_d^L}{(2d-1)!}q^{2d},\\
F^C &=& \sum_{d=1}^{\infty}\frac{16\bar{N}_d^C}{(d-1)!}q^d
\end{eqnarray*}
be two generating series, which encode information of $\bar{N}_d^L$, $\bar{N}_d^C$, respectively.

It is easy to deduce from \eqref{eqn:simprec} that $F^L$ and $F^C$ both satisfy the same differential equation:
\begin{equation}\label{eqn:key}
\Big(q\frac{d}{dq}-1\Big)\Big(q\frac{d}{dq}-2\Big)F^{\bullet}=\frac{\big(q\frac{d}{dq}\big)^2F^{\bullet}}{4}\times\Big(q\frac{d}{dq}-1\Big)F^{\bullet}
\end{equation}
where $\bullet$ stands for $L$ or $C$.

In the rest of this subsection, we will solve differential equation \eqref{eqn:key} and get closed formulae for $\bar{N}_d^L$ and $\bar{N}_d^C$. 

Let 
\[A^{\bullet}=\Big(\frac{d}{dq}-\frac{1}{q}\Big)F^{\bullet}.\]
Our main theorem can be stated as follows:
\begin{thm}\label{thm:closefm}
\begin{eqnarray*}
A^L\exp\Bigg(W\Big(\frac{A^L}{2\sqrt{-1}}\Big)+W\Big(-\frac{A^L}{2\sqrt{-1}}\Big)\Bigg) & = & q,\\
A^{C}\exp\Bigg(2W\Big(-\frac{A^{C}}{8}\Big)\Bigg) & = & 16q,
\end{eqnarray*}
where 
\begin{equation}\label{eqn:lamb}
W(x):=\sum_{k=1}^{\infty}\frac{k^{k-1}}{k!}x^k
\end{equation}
is the famous Lambert W-function.
\end{thm}

Using Lagrangian Inversion Theorem, it is easy to deduce that 

\begin{cor}\label{cor:explfm}
\begin{eqnarray*}
\bar{N}_{d+1}^L & = & \frac{(2d)!}{2d+1}\sum_{s=1}^d\sum_{a_1+a_2+\cdots a_s=d, \atop a_i>0} \frac{(-1)^{d-s}(2d+1)^s}{s!}\prod_{i=1}^s \frac{a_i^{2a_i-1}}{(2a_i)!},\\
\bar{N}_{d+2}^C & = & \frac{d!}{d+1}\sum_{s=1}^d\sum_{a_1+a_2+\cdots a_s=d, \atop a_i>0}\frac{(-1)^{d-s}2^{d+s}(d+1)^s}{s!}\prod_{i=1}^s\frac{a_i^{a_i-1}}{a_i!}.
\end{eqnarray*}
\end{cor}

\begin{rmk}
In \cite{RW}, the authors also deduce the same formula for $\bar{N}_d^L$ using a totally different method. We will adopt a purely combinatorial approach which can be used to deduce $\bar{N}_d^L$ and $\bar{N}_d^C$ in a unified way.
\end{rmk}

\begin{proof}[Proof of Theorem \ref{thm:closefm}]
Let 
\[B^{\bullet}:=\frac{q\frac{d}{dq}}{4}F^{\bullet}.\]
A simple calculation can show that \eqref{eqn:key} is equivalent to 
\[q\frac{d}{dq}\Big\{log\Big(\frac{A^{\bullet}}{q}\Big)-B^{\bullet}\Big\}=0.\]
So we must have 
\begin{equation*}
A^{\bullet}=M^{\bullet}q\exp(B^{\bullet}),
\end{equation*}
where $M^{L}$, $M^{C}$ are two constants that can be determined by the initial values of $\bar{N}_d^{L}$, $\bar{N}_d^{C}$ respectively. By using Gathmann's program \cite{GROWI}, we can compute that 
\[\bar{N}_1^L=1,\qquad \bar{N}_2^C=1.\]
Then we deduce that
\[M^{L}=1,\qquad M^{C}=16.\]
Now we want to find two functions $f^{L}$, $f^{C}$ such that 
\[f^{\bullet}(A^{\bullet})=B^{\bullet}\]
Taking derivative on both sides with respect to $q$, we get
\[\frac{df^{\bullet}}{dA^{\bullet}}=\frac{
\frac{dB^{\bullet}}{dq}}{\frac{dA^{\bullet}}{dq}}.\]
Using \eqref{eqn:key}, it is easy to check that
\[\frac{A^{\bullet}
\frac{dB^{\bullet}}{dq}}{\frac{dA^{\bullet}}{dq}}=\frac{B^{\bullet}+\frac{qA^{\bullet}}{2}}{1+B^{\bullet}+\frac{qA^{\bullet}}{4}}.\]
Replacing $q$ by $A^{\bullet}$ and $B^{\bullet}$, we have
\[A^{\bullet}\frac{df^{\bullet}}{dA^{\bullet}}=\frac{f^{\bullet}+\frac{(A^{\bullet})^2}{2M^{\bullet}}e^{-f^{\bullet}}}{1+f^{\bullet}+\frac{(A^{\bullet})^2}{4M^{\bullet}}e^{-f^{\bullet}}}\]
Now using Lemma \ref{lem:ode} below and some initial values of $N_d^{\bullet}$, we deduce that
\begin{eqnarray*}
f^{L}(A^{L}) & = & -W\Big(\frac{A^L}{2\sqrt{-1}}\Big)-W\Big(-\frac{A^L}{2\sqrt{-1}}\Big)\\
f^{C}(A^{C}) & = & -2W\Big(-\frac{A^{C}}{8}\Big)
\end{eqnarray*}

The proof of Theorem \ref{thm:closefm} is complete.
\end{proof}

\begin{lem}\label{lem:ode}
Let 
\[f(x)=\sum_{k=0}^{\infty}a_kx^k\]
be a formal power series, which satisfies the following ordinary differential equation:
\[x\frac{d}{dx}f=\frac{f+\frac{x^2}{2}e^{-f}}{1+f+\frac{x^2}{4}e^{-f}}.\]
Then $f$ can be written as
\[-W(c_1x)-W(c_2x),\]
where $W(x)$
is the Lambert W function defined in \eqref{eqn:lamb}, and $c_1$, $c_2$ are two complex numbers that satisfy $c_1c_2=\frac{1}{4}$.
\end{lem}

\begin{proof}
By comparing the coefficients on both sides of the ordinary differential equation, one can see that $a_0=0$ and all the other coefficients can be uniquely determined by $a_1$. Since the coefficient $-c_1-c_2$ in $-W(c_1x)-W(c_2x)$ can be arbitrary, we only need to show that $-W(c_1x)-W(c_2x)$ satisfies the above equation. Then the uniqueness implies that they are all the solutions.

In order to show that 
\[f(x)=-W(c_1x)-W(c_2x),\qquad c_1c_2=\frac{1}{4}\]
is a solution. It is equivalent to check that
\[x\frac{d}{dx}\Big(f+\frac{f^2}{2}-\frac{x^2}{4}e^{-f}\Big)=f.\]
Using the well-known fact that 
\begin{equation}\label{eqn:Lambinv}
We^{-W}=x    
\end{equation}
and $c_1c_2=\frac{1}{4}$, we can show that
\[f+\frac{f^2}{2}-\frac{x^2}{4}e^{-f}=-W(c_1x)+\frac{W(c_1x)^2}{2}-W(c_2x)+\frac{W(c_2x)^2}{2}\]
Now using \eqref{eqn:Lambinv}, it is easy to see that
\[x\frac{d}{dx}\Big(-W(c_ix)+\frac{W(c_ix)^2}{2}\Big)=-W(c_ix),\qquad i=1,2.\]
The Lemma \ref{lem:ode} then follows.
\end{proof}

\newpage
\bibliography{universal-BIB}
\bibliographystyle{amsxport}
\end{document}